\documentclass[english]{article}
\usepackage[T1]{fontenc}
\usepackage[utf8]{inputenc}
\usepackage{amsthm}
\usepackage{amsmath}
\usepackage{amssymb}

\makeatletter
 \theoremstyle{definition}
 \newtheorem{defn}{\protect\definitionname}
  \theoremstyle{plain}
  \newtheorem{prop}{\protect\propositionname}
  \theoremstyle{plain}
  \newtheorem{lem}{\protect\lemmaname}
  \theoremstyle{remark}
  \newtheorem{rem}{\protect\remarkname}
  \theoremstyle{plain}
  \newtheorem{thm}{\protect\theoremname}
  \theoremstyle{plain}
 \newtheorem{question}{\protect\questionname}
 \theoremstyle{plain}
 \newtheorem{cor}{\protect\corollaryname}

\makeatother

\usepackage{babel}
  \providecommand{\definitionname}{Definition}
   \providecommand{\theoremname}{Theorem}
  \providecommand{\lemmaname}{Lemma}
  \providecommand{\questionname}{Question}
  \providecommand{\propositionname}{Proposition}
  \providecommand{\remarkname}{Remark}
  \providecommand{\corollaryname}{Corolary}

\begin{document}
\global\long\def\norm#1{\left\Vert #1\right\Vert }
 \global\long\def\Lip{\operatorname{Lip}}
 \global\long\def\const{\operatorname{const}}
 \global\long\def\ZZ{\mathbb{Z}}
 \global\long\def\RR{\mathbb{R}}

\title{Bi-Lipschitz bijections of $\ZZ$}

\author{Itai Benjamini and Alexander Shamov}
\date{February 4, 2014}
\maketitle
\begin{abstract}
It is shown that
every bi-Lipschitz bijection from $\ZZ$ to itself is at a bounded $L_{\infty}$ distance from either  the identity or
the reflection.  We then comment on  the group-theoretic properties of the action of bi-Lipschitz bijections.
\end{abstract}

\section{Introduction}

\begin{defn}
A bi-Lipschitz bijection between two metric spaces $\left(X,\rho_{X}\right)$
and $\left(Y,\rho_{Y}\right)$ is a bijective map $f:X\to Y$, such
that there are $0<C_{1}\le C_{2}<+\infty$, such that for all $x_{1},x_{2}\in X$
\[
C_{1}\rho_{X}\left(x_{1},x_{2}\right)\le\rho_{Y}\left(f\left(x_{1}\right),f\left(x_{2}\right)\right)\le C_{2}\rho_{X}\left(x_{1},x_{2}\right).
\]

\end{defn}
Recall the definition of the Lipschitz  constant of a map:
\[
\norm f_{\Lip}:=\sup_{x_{1}\neq x_{2}}\frac{\rho_{Y}\left(f\left(x_{1}\right),f\left(x_{2}\right)\right)}{\rho_{X}\left(x_{1},x_{2}\right)}.
\]

A map $f$ is Lipschitz if and only if $\norm f_{\Lip}$ is finite,
and bi-Lipschitz if and only if it is bijective and both $\norm f_{\Lip}$ and
$\norm{f^{-1}}_{\Lip}$ are finite.

While the real line  $\RR$ admits a large family of bi-Lipschitz bijections, e.g. including any increasing function with derivative bounded away from $0$
and $\infty$,   bi-Lipschitz bijections of  $\ZZ$ turn out to be much more rigid. Namely, we have

\begin{thm} \label{thm:main}
Let $f:\ZZ\to\ZZ$ be a bi-Lipschitz bijection ($\ZZ$
is equipped with its usual metric, namely $\rho\left(x,y\right):=\left|x-y\right|$).
Then either
\[
\sup_{x\in\ZZ}\left|f\left(x\right)-x\right|<+\infty
\]
or
\[
\sup_{x\in\ZZ}\left|f\left(x\right)+x\right|<+\infty.
\]
More precisely,
\[
f\left(x\right)=\pm x+\const+r\left(x\right),
\]
\[
\left|r\left(x\right)\right|\le\norm f_{\Lip}\norm{f^{-1}}_{\Lip}.
\]

\end{thm}

This result extends to spaces that are bi-Lipschitz
isomorphic to $\ZZ$, like, for instance, products $\ZZ\times G$
with a finite graph $G$, equipped with the graph metric.
\medskip

The reason for different behavior of $\ZZ$ vs. $\RR$ is that
unlike  $\RR$,  $\ZZ$ cannot be {\it ``squeezed and stretched'' }.
In the proof below one of the arguments is a cardinality estimate.
It is quite obvious that this argument fails in the continuum,
and indeed for $\RR$ the statement is just wrong. However,
the analogy is restored if we equip our space with a measure and require
the bijection to be measure preserving. This motivates the
following
\begin{question}
Let $f:\ZZ^2 \to\ZZ^2$ be a bi-Lipschitz bijection.
Can it be extended to a bi-Lipschitz Lebesgue measure preserving bijection
$g:\RR^2 \to\RR^2$?
\end{question}

Note that the two dimensional grid $\ZZ^2$ admits many bi-Lipschitz bijections.
For example, let $g: \ZZ \to \ZZ$ be a Lipschitz function. Then $F(x,y) := (x, y+g(x))$ is a bi-Lipschitz bijection of
$\ZZ^2$. This shows that a naive generalization of Theorem \ref{thm:main} fails for $\ZZ^2$: not every bi-Lipschitz bijection is at
a bounded distance from an isometry.
\medskip

For background on metric geometry see e.g. \cite{BBI}. The group of bijections from $\ZZ$ to $\ZZ$ within a bounded
$L_\infty$ distance to the identity recently appeared in \cite{JM}.

\section{Proof of Theorem \ref{thm:main}}

The key to the result is to understand how the image sets $f\left(\left(-\infty,x\right]\right)$
may look like.
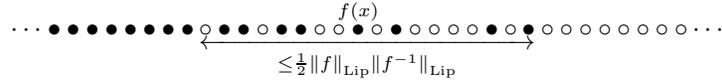
\[
\cdots\bullet\bullet\bullet\bullet\bullet\bullet\bullet\bullet\,\underset{\le\frac{1}{2}\norm f_{\Lip}\norm{f^{-1}}_{\Lip}}{\underleftrightarrow{\circ\bullet\bullet\circ\bullet\bullet\circ\overset{f\left(x\right)}{\circ\bullet\circ}\bullet\circ\circ\circ\circ\bullet\circ\,\bullet}}\circ\circ\circ\circ\circ\circ\circ\circ\cdots
\]
The ``picture'' above illustrates what we are going to prove.
$\circ$'s are used to denote $y\in\mathbb{Z}$ such that $y\notin f\left(\left(-\infty,x\right]\right)$,
and $\bullet$'s for $y\in f\left(\left(-\infty,x\right]\right)$.

In the sequel we denote the constant $\norm f_{\Lip}\norm{f^{-1}}_{\Lip}$
by $C$.
\begin{lem}  \label{lem:main}
One of the following two cases occurs: either
\[
\left(-\infty,\left\lfloor f\left(x\right)-C/2\right\rfloor \right]\subset f\left(\left(-\infty,x\right]\right)\subset\left(-\infty,\left\lfloor f\left(x\right)+C/2\right\rfloor \right]
\]
or
\[
\left[\left\lceil f\left(x\right)+C/2\right\rceil ,+\infty\right)\subset f\left(\left(-\infty,x\right]\right)\subset\left[\left\lceil f\left(x\right)-C/2\right\rceil ,+\infty\right).
\]
for all $x \in \ZZ$.
\end{lem}
\begin{proof}
Let $y\neq f\left(x\right)$ be such that $y\in f\left(\left(-\infty,x\right]\right)$
and $y+1\notin f\left(\left(-\infty,x\right]\right)$ (i.e. $y$ is
the position of a ``$\bullet\circ$'' on the ``picture''). Then
since $y\in f\left(\left(-\infty,x\right]\right)$, it follows that
$f^{-1}\left(y\right)<x$. In the same way since $y+1\notin f\left(\left(-\infty,x\right]\right)$,
we have $f^{-1}\left(y+1\right)>x$. From the Lipschitz property of
$f$ it follows that
\[
x-f^{-1}\left(y\right)\ge\frac{\left|f\left(x\right)-y\right|}{\norm f_{\Lip}},
\]
\[
f^{-1}\left(y+1\right)-x\ge\frac{\left|f\left(x\right)-y-1\right|}{\norm f_{\Lip}}.
\]
Therefore,
\[
f^{-1}\left(y+1\right)-f^{-1}\left(y\right)\ge\frac{2\left|f\left(x\right)-y-\frac{1}{2}\right|}{\norm f_{\Lip}}.
\]
Now from the Lipschitz property of $f^{-1}$ it follows that
\[
1=\left(y+1\right)-y\ge\frac{f^{-1}\left(y+1\right)-f^{-1}\left(y\right)}{\norm{f^{-1}}_{\Lip}}\ge\frac{2\left|f\left(x\right)-y-\frac{1}{2}\right|}{\norm f_{\Lip}\norm{f^{-1}}_{\Lip}}.
\]
In other words, the distance between $f\left(x\right)$ and any ``$\bullet\circ$''
is bounded by $\frac{1}{2}C$.

The same argument also applies to ``$\circ\bullet$'': just replace
$y+1$ by $y-1$ everywhere.

This proves that the characteristic function of the set $f\left(\left(-\infty,x\right]\right)$
does not change outside the region
\[
\left[f\left(x\right)-\frac{1}{2}C,f\left(x\right)+\frac{1}{2}C\right].
\]
Now since both $f\left(\left(-\infty,x\right]\right)$ and its complement
$\ZZ\setminus f\left(\left(-\infty,x\right]\right)=f\left(\left[x+1,+\infty\right)\right)$
must be infinite, only two possibilities are left: either $f\left(\left(-\infty,x\right]\right)$ is unbounded
from below or it is unbounded from above, which obviously corresponds
to the two possible conclusions of the lemma.\end{proof}
\begin{rem}
Actually, by a little more careful application of the same argument
one can show that the width of the region where the characteristic
function of $f\left(\left(-\infty,x\right]\right)$ is nonconstant
is bounded by $\frac{1}{2}C$.\end{rem}
\begin{proof}
[Proof of Theorem \ref{thm:main}] Let's assume that the images in Lemma \ref{lem:main} are
unbounded from below (the other case can be treated analogously).
Let $x_{1},x_{2}\in\mathbb{Z}$ be such that $x_{2}-x_{1}>C$. Then
\[
f\left(\left(x_{1},x_{2}\right]\right)=f\left(\left(-\infty,x_{2}\right]\right)\setminus f\left(\left(-\infty,x_{1}\right]\right)\subset
\]
\[
\subset\left(-\infty,f\left(x_{2}\right)+C/2\right]\setminus\left(-\infty,f\left(x_{1}\right)-C/2\right]=
\]
\[
=\left(f\left(x_{1}\right)-C/2,f\left(x_{2}\right)+C/2\right].
\]
In the same way
\[
f\left(\left(x_{1},x_{2}\right]\right)\supset\left(f\left(x_{1}\right)+C/2,f\left(x_{2}\right)-C/2\right].
\]
Since $f$ is a bijection, the cardinality of $f\left(\left(x_{1},x_{2}\right]\right)$
must be $x_{2}-x_{1}$. Therefore,
\[
f\left(x_{2}\right)-f\left(x_{1}\right)-C\le x_{2}-x_{1}\le f\left(x_{2}\right)-f\left(x_{1}\right)+C.
\]
Now if we fix $x_{1}<0$ and vary $x_{2}$, we see that for $x$ in the interval $\left[x_1, +\infty\right)$
\[
\left| f\left(x\right) - x - \const_{x_1} \right| \le C.
\]
Note that $x_{1}$ can be arbitrary and the range of possible values of $\const_{x_1}$ is bounded independently of $x_1$
(e.g. $\left| \const_{x_1} \right| \le \left|f(0)\right| + C$),
therefore the bound holds on the whole $\ZZ$.\end{proof}

\section{Corollaries}

As pointed out by the referee, our result implies that there is a remarkable difference between $\ZZ$ and higher-dimensional
lattices in terms of the group-theoretic properties of the action of bi-Lipschitz bijections. In particular:

\begin{cor} \label{cor:propT}
 The group of bi-Lipschitz bijections of $\ZZ$ does not contain an infinite countable subgroup with property (T).
\end{cor}

\begin{proof}
 The fact that the wobbling group of $\ZZ$ -- i.e. the group of bijections that have finite $\ell^\infty$ distance from the identity --
 does not contain a countable property (T) subgroup follows from Theorem 4.1 in \cite{JS}. On the other hand, by our result, the wobbling
 group of $\ZZ$ is an index $2$ subgroup of the group of bi-Lipschitz bijections.
\end{proof}

Note that Corollary \ref{cor:propT} fails for $\ZZ^d, d \ge 3$, since $\operatorname{SL}(d, \ZZ), d \ge 3$ has property (T) and
acts faithfully on $\ZZ^d$ by bi-Lipschitz bijections. We do not know what happens in the $d=2$ case.

\begin{question}
 Does Corollary \ref{cor:propT} hold for $\ZZ^2$?
\end{question}

Another corollary concerns an amenability-like property:

\begin{cor} \label{cor:invMean}
 There is a bi-Lipschitz invariant mean (i.e. finitely additive probability measure) on $\ZZ$.
\end{cor}

\begin{proof}
 From Lemma \ref{lem:main} it follows that the sets $A_n := [-n, n]$ form a F\o lner sequence for the action of bi-Lipschitz bijections -- i.e.
 for any particular bi-Lipschitz bijection $f$ we have
 \[
  \frac{\left|f(A_n) \cap A_n\right|}{\left| A_n\right|} \to 1, n \to \infty
 \]
 Therefore, an invariant mean can be obtained by a standard argument, as a limiting point of the sequence of uniform measures on $A_n$
 with respect to the weak-$\ast$ topology of $(\ell^\infty)^\ast$.
\end{proof}

On the other hand:

\begin{prop}
 Corollary \ref{cor:invMean} fails for $\ZZ^2$.
\end{prop}

\begin{proof}
 Let $\mu$ be a bi-Lipschitz invariant mean on $\ZZ^2$. Then the standard action of $\operatorname{SL}(2, \ZZ)$ on
 $\ZZ^2 \setminus \{0\}$ preserves the mean $\mu$ restricted to $\ZZ \setminus \{0\}$. This is impossible, since
 $\operatorname{SL}(2, \ZZ)$ is nonamenable and acts on $\ZZ^2 \setminus \{0\}$ with amenable stabilizers.
\end{proof}

\section{Acknowledgements}

The authors wish to thank the referee for suggesting Corollaries \ref{cor:propT} and \ref{cor:invMean} and the surrounding discussion.


\begin{thebibliography}{BK}

\bibitem{BBI}
Dmitri Burago, Yuri Burago, and Sergei Ivanov,
A Course in Metric Geometry.
American Mathematical Society (2001).

\bibitem{JM}
Kate Juschenko and Nicolas Monod,
Cantor systems, piecewise translations and simple amenable groups.
{\emph Annals of Mathematics} {\bf 2} (2013), 775-787

\bibitem{JS}
Kate Juschenko and Mikael de la Salle,
Invariant means for the wobbling group.
{\emph Bull. Belg. Math. Soc. Simon Stevin} {\bf 22} (2015), no. 2, 281290.

\end{thebibliography}
\end{document}